\documentclass[12pt,reqno]{article}

\usepackage{amsmath,amsthm,amsfonts,amssymb,latexsym,mathrsfs,color,epic}
\usepackage[colorlinks=true,
linkcolor=webgreen,
filecolor=webbrown,
citecolor=webgreen]{hyperref}

\definecolor{webgreen}{rgb}{0,.5,0}
\definecolor{webbrown}{rgb}{.6,0,0}

\newtheorem{thm}{Theorem}[section]

\newtheorem{prop}[thm]{Proposition}

\numberwithin{equation}{section}

\newcommand{\cs}{{\mathfrak{S}}}

\def \des{{\rm des}\,}
\def \odes{{\rm odes}\,}
\def \edes{{\rm edes}\,}
\def \oasc{{\rm oasc}\,}
\def \easc{{\rm easc}\,}

\def \Odes{{\rm Odes}\,}
\def \Edes{{\rm Edes}\,}
\def \Oasc{{\rm Oasc}\,}
\def \Easc{{\rm Easc}\,}

\def \Orb{{\rm Orb}\,}

\newcommand{\lrf}[1]{\left\lfloor #1\right\rfloor}

\title{A new class of refined Eulerian polynomials
\thanks{Partially supported by the National Natural Science Foundation of China (Grant No.11526044) and
the Doctoral Scientific Research Starting Foundation of Liaoning Province(No.20170520451).
\newline\hspace*{5mm}
   {\it Email address:}\quad  sunhua@dlou.edu.cn}}
\author{Hua Sun}
\date{\footnotesize College of Sciences, Dalian Ocean University, Dalian 116023, P.R. China}
\begin{document}

\maketitle

\begin{abstract}
In this note we introduce a new class of refined Eulerian polynomials defined by
$$A_n(p,q)=\sum_{\pi\in\cs_{n}}p^{\odes(\pi)}q^{\edes(\pi)},$$
where $\odes(\pi)$ and $\edes(\pi)$ enumerate the number of descents of permutation $\pi$ in odd and even positions, respectively.
We show that the refined Eulerian polynomials $A_{2k+1}(p,q),k=0,1,2,\ldots,$ and $(1+q)A_{2k}(p,q),k=1,2,\ldots,$ have a nice symmetry property.
\end{abstract}

Key words:\quad  odd descents; even descents; Eulerian polynomials; $\gamma$-positivity

Mathematics Subject Classifications:\quad 05A05, 05A15, 05A19

\section{Introduction}\hspace*{\parindent}
Let $f(q)=a_{r}q^r+\cdots+a_{s}q^s(r\leq s)$, with $a_{r}\neq 0$ and $a_{s}\neq 0$, be a real polynomial.
The polynomial $f(q)$ is {\it palindromic} if $a_{r+i}=a_{s-i}$ for any $i$.
Following Zeilberger~\cite{Zei88}, define the {\it darga} of $f(q)$ to be $r+s$.
The set of all palindromic polynomials of darga $n$ is a vector space~\cite{SWZ15} with gamma basis
$$\Gamma_{n}:=\{q^{i}(1+q)^{n-2i}|0\leq i\leq \lrf{n/2}\}.$$

Let $f(p,q)$ be a nonzero bivariate polynomial. The polynomial $f(p,q)$ is
{\it palindromic of darga $n$} if it satisfies the following two equations
\begin{gather*}
f(p,q)=f(q,p),\\
f(p,q)=(pq)^{n}f(1/p,1/q).
\end{gather*}
See~\cite{ABERS17} for details.
It is known~\cite{Lin16} that the set of all palindromic bivariate polynomials of darga $n$ is a vector space with gamma basis
$$\mathcal{B}_{n}:=\{(pq)^{i}(p+q)^{j}(1+pq)^{n-2i-j}|i,j\geq0,2i+j\leq n\}.$$

Let $\cs_{n}$ denote the set of all permutations of the set $[n]:=\{1,2,\ldots,n\}$. For a permutation
$\pi=\pi_{1}\pi_{2}\cdots\pi_{n}\in\cs_n$, an index $i\in[n-1]$ is a {\it descent} of $\pi$ if $\pi_i>\pi_{i+1}$,
and $\des(\pi)$ denotes the number of descents of $\pi$.
The classic Eulerian polynomial is defined as the generating
polynomial for the statistic des over the set $\cs_{n}$, i.e.,
$$A_{n}(q)=\sum_{\pi\in\cs_{n}}q^{\des(\pi)}.$$
Foata and Sch\"{u}tzenberger~\cite{FS70} proved that the Eulerian polynomial $A_{n}(q)$ can be expressed in terms of the gamma basis $\Gamma_{n}$
with nonnegative integer coefficients.
A polynomial with nonnegative coefficients under the gamma basis $\Gamma_{n}$
is palindromic and unimodal~\cite{Pet15}.

Ehrenborg and Readdy~\cite{ER16} studied the number of ascents in odd position on $0,1$-words. We define similar statistics on permutations. For a permutation $\pi\in\cs_{n}$,
an index $i\in[n-1]$ is an {\it odd descent} of $\pi$ if $\pi_i>\pi_{i+1}$ and $i$ is odd,
an {\it even descent} of $\pi$ if $\pi_i>\pi_{i+1}$ and $i$ is even,
an {\it odd ascent} of $\pi$ if $\pi_i<\pi_{i+1}$ and $i$ is odd,
an {\it even ascent} of $\pi$ if $\pi_i<\pi_{i+1}$ and $i$ is even.
Let $\Odes(\pi)$, $\Edes(\pi)$, $\Oasc(\pi)$ and $\Easc(\pi)$
denote the set of all odd descents, even descents, odd ascents and even ascents of $\pi$, respectively.
The corresponding cardinalities are $\odes(\pi)$, $\edes(\pi)$, $\oasc(\pi)$ and $\easc(\pi)$, respectively.
Note that we can also define the above four statistics on words of length $n$.
The joint distribution of odd and even descents on $\cs_n$ is denoted by $A_n(p,q)$, i.e.,
$$A_n(p,q)=\sum_{\pi\in\cs_{n}}p^{\odes(\pi)}q^{\edes(\pi)}.$$

The polynomial $A_n(p,q)$ is a bivariate polynomial of degree $n-1$. The monomial with degree $n-1$ is
$p^{\lrf{n/2}}q^{\lrf{(n-1)/2}}$ only. If $p=q$, then $A_n(q,q)=A_n(q)$ is the classic Eulerian polynomial.
Thus $A_n(p,q)$, $n=1,2,\ldots,$ can be seen as a class of refined Eulerian polynomials. For example, we have
\begin{align*}
A_1(p,q)&=1,\\
A_2(p,q)&=1+p,\\
A_3(p,q)&=1+2p+2q+pq,\\
A_4(p,q)&=1+6p+5q+5p^2+6pq+p^{2}q,\\
A_5(p,q)&=1+13p+13q+16p^2+34pq+16q^2+13p^{2}q+13pq^{2}+p^{2}q^{2},\\
A_6(p,q)&=1+29p+28q+89p^2+152pq+61q^2+61p^{3}+152p^{2}q\\
&~~~+89pq^{2}+28p^{3}q+29p^{2}q^{2}+p^{3}q^{2}.
\end{align*}

For convenience, we denote
\begin{equation*}
\widetilde{A}_n(p,q)=
\begin{cases}
A_n(p,q) & \text{if $n=2k+1$,}\\
(1+q)A_n(p,q) & \text{if $n=2k$.}
\end{cases}
\end{equation*}
Our main result is the following
\begin{thm}\label{pal}
For any $n=1,2,\ldots,$ the polynomial $\widetilde{A}_{n}(p,q)$ is palindromic of darga $\lrf{\frac{n}{2}}$.
\end{thm}

In the next section we give a proof of Theorem~\ref{pal}.
In Section 3 we study the case $q=1$ and the case $p=1$,
the polynomials $A_n(p,1)$ and $A_n(1,q)$ are the generating functions for the statistics odes and edes over the set $\cs_n$, respectively.
In the last section, we propose a conjecture that $\widetilde{A}_{n}(p,q)$ can be expressed in terms of the gamma basis $\mathcal{B}_{\lrf{\frac{n}{2}}}$
with nonnegative integer coefficients.

\section{The proof of Theorem~\ref{pal}}\hspace*{\parindent}
Let $\pi=\pi_{1}\pi_{2}\cdots\pi_{n}\in\cs_n$, we define the {\it reversal} $\pi^r$ of $\pi$ to be
$$\pi^r:=\pi_{n}\pi_{n-1}\cdots\pi_{1},$$
the {\it complement} $\pi^c$ of $\pi$ to be
$$\pi^c:=(n+1-\pi_{1})(n+1-\pi_{2})\cdots(n+1-\pi_{n}),$$
and the {\it reversal-complement} $\pi^{rc}$ of $\pi$ to be
$$\pi^{rc}:=(\pi^{c})^r=(\pi^{r})^c.$$

If $i$ is a descent of $\pi$, then $i$ is an ascent of $\pi^c$ and if $i$ is an ascent of $\pi$, then $i$ is a descent of $\pi^c$.
In other words, $\odes(\pi)+\odes(\pi^c)=\lrf{\frac{n}{2}}$ and $\edes(\pi)+\edes(\pi^c)=\lrf{\frac{n-1}{2}}$.
Then
\begin{equation*}
\begin{split}
A_{n}(p,q)&=\sum_{\pi\in\cs_{n}}p^{\odes(\pi)}q^{\edes(\pi)}=\sum_{\pi\in\cs_{n}}p^{\lrf{\frac{n}{2}}-\odes(\pi^c)}q^{\lrf{\frac{n-1}{2}}-\edes(\pi^c)}\\
&=p^{\lrf{\frac{n}{2}}}q^{\lrf{\frac{n-1}{2}}}\sum_{\pi\in\cs_{n}}\left(\frac{1}{p}\right)^{\odes(\pi^c)}\left(\frac{1}{q}\right)^{\edes(\pi^c)}
=p^{\lrf{\frac{n}{2}}}q^{\lrf{\frac{n-1}{2}}}A_{n}\left(\frac{1}{p},\frac{1}{q}\right).
\end{split}
\end{equation*}
Specially, for any $k=1,2,\ldots$, we have $A_{2k}(p,q)=p^{k}q^{k-1}A_{2k}(1/p,1/q)$ and for any $k=0,1,2,\ldots$, we have $A_{2k+1}(p,q)=(pq)^{k}A_{2k+1}(1/p,1/q)$.

It can be derived that $i$ is a descent of $\pi$ if and only if $i$ is an ascent of $\pi^c$.
It is also easy to see that $i$ is a descent of $\pi$ if and only if $n-i$ is an ascent of $\pi^r$. Then, given a permutation $\pi=\pi_{1}\pi_{2}\cdots\pi_{2k+1}\in\cs_{2k+1}$,
\begin{center}
$i$ is a descent of $\pi$ if and only if $2k+1-i$ is a descent of $\pi^{rc}$.
\end{center}
Specially, $i$ is an odd descent of $\pi$ if and only if $2k+1-i$ is an even descent of $\pi^{rc}$,
and $i$ is an even descent of $\pi$ if and only if $2k+1-i$ is an odd descent of $\pi^{rc}$. So we have
$$\widetilde{A}_{2k+1}(p,q)=\sum_{\pi\in\cs_{2k+1}}p^{\odes(\pi)}q^{\edes(\pi)}=\sum_{\pi\in\cs_{2k+1}}p^{\edes(\pi^{rc})}q^{\odes(\pi^{rc})}=\widetilde{A}_{2k+1}(q,p).$$
Thus for any $k=1,2,\ldots,$ the polynomial $\widetilde{A}_{2k+1}(p,q)$ is palindromic of darga $k$.

In addition,
$$\widetilde{A}_{2k}(p,q)=(1+q)p^{k}q^{k-1}A_{2k}\left(\frac{1}{p},\frac{1}{q}\right)=\left(1+\frac{1}{q}\right)p^{k}q^{k}A_{2k}\left(\frac{1}{p},\frac{1}{q}\right)
=(pq)^{k}\widetilde{A}_{2k+1}\left(\frac{1}{p},\frac{1}{q}\right).$$

The last part is to prove that $\widetilde{A}_{2k}(p,q)=\widetilde{A}_{2k}(q,p)$, that is,
$$\sum_{\pi\in\cs_{2k}}p^{\odes(\pi)}[q^{\edes(\pi)}+q^{\edes(\pi)+1}]=\sum_{\pi\in\cs_{2k}}q^{\odes(\pi)}[p^{\edes(\pi)}+p^{\edes(\pi)+1}].$$

Let $\cs'_{2k}=\left\{\pi(2k+1),\pi0|\pi\in\cs_{2k}\right\}$, $\cs''_{2k}=\left\{(2k+1)\pi,0\pi|\pi\in\cs_{2k}\right\}$,
and let $\pi=\pi_{1}\pi_{2}\cdots\pi_{2k}\in\cs_{2k}$. Define a map $\psi:\cs'_{2k}\rightarrow\cs''_{2k}$ by
\begin{equation*}
\psi(\pi x)=
\begin{cases}
(2k+1)(2k+1-\pi_{2k})(2k+1-\pi_{2k-1})\cdots(2k+1-\pi_{1}) & \text{if $x=0$,}\\
0(2k+1-\pi_{2k})(2k+1-\pi_{2k-1})\cdots(2k+1-\pi_{1}) & \text{if $x=2k+1$.}
\end{cases}
\end{equation*}

Given a permutation $\pi\in\cs_{2k}$, it is no hard to see that
\begin{align*}
&\odes(\pi(2k+1))=\odes(\pi),&\edes(\pi(2k+1))=\edes(\pi),\\
&\odes(\pi0)=\odes(\pi),&\edes(\pi0)=\edes(\pi)+1,\\
&\odes((2k+1)\pi)=\edes(\pi)+1,&\edes((2k+1)\pi)=\odes(\pi),\\
&\odes(0\pi)=\edes(\pi),&\edes(0\pi)=\odes(\pi).
\end{align*}
Thus
\begin{align*}
&\odes(\psi(\pi(2k+1)))=\odes(0\pi^{rc})=\edes(\pi^{rc}),\\
&\edes(\psi(\pi(2k+1)))=\edes(0\pi^{rc})=\odes(\pi^{rc}),\\
&\odes(\psi(\pi0))=\odes((2k+1)\pi^{rc})=\edes(\pi^{rc})+1,\\
&\edes(\psi(\pi0))=\edes((2k+1)\pi^{rc})=\odes(\pi^{rc}).
\end{align*}
Obviously, the map $\psi$ is an involution. Then
\begin{equation*}
\begin{split}
&\sum_{\pi\in\cs_{2k}}p^{\odes(\pi)}[q^{\edes(\pi)}+q^{\edes(\pi)+1}]\\
&=\sum_{\pi\in\cs_{2k}}p^{\odes(\pi(2k+1))}q^{\edes(\pi(2k+1))}+\sum_{\pi\in\cs_{2k}}p^{\odes(\pi0)}q^{\edes(\pi0)}\\
&=\sum_{\pi\in\cs_{2k}}p^{\odes(\psi(\pi(2k+1)))}q^{\edes(\psi(\pi(2k+1)))}+\sum_{\pi\in\cs_{2k}}p^{\odes(\psi(\pi0))}q^{\edes(\psi(\pi0))}\\
&=\sum_{\pi\in\cs_{2k}}p^{\edes(\pi^{rc})}q^{\odes(\pi^{rc})}+\sum_{\pi\in\cs_{2k}}p^{\edes(\pi^{rc})+1}q^{\odes(\pi^{rc})}\\
&=\sum_{\pi\in\cs_{2k}}q^{\odes(\pi)}[p^{\edes(\pi)}+p^{\edes(\pi)+1}].
\end{split}
\end{equation*}
Thus for any $k=1,2,\ldots,$ the polynomial $\widetilde{A}_{2k}(p,q)$ is palindromic of darga $k$.
This completes the proof.

\section{The case $p=1$ and the case $q=1$}\hspace*{\parindent}
If $q=1$, the polynomial $A_n(p,1)$ is the generating function for the statistic odes over the set $\cs_n$,
and if $p=1$, the polynomial $A_n(1,q)$ is the generating function for the statistic edes  over the set $\cs_n$.
More precisely, we have
\begin{prop}
Let $n$ be a positive integer. Then
\begin{equation}\label{equ1}
\sum_{\pi\in\cs_{n}}p^{\odes(\pi)}=A_n(p,1)=\frac{n!}{2^{\lrf{\frac{n}{2}}}}(1+p)^{\lrf{\frac{n}{2}}},
\end{equation}
and
\begin{equation}\label{equ2}
\sum_{\pi\in\cs_{n}}q^{\edes(\pi)}=A_n(1,q)=\frac{n!}{2^{\lrf{\frac{n-1}{2}}}}(1+q)^{\lrf{\frac{n-1}{2}}}.
\end{equation}
\end{prop}

\begin{proof}
It is easy to verify that the equalities~\ref{equ1} and~\ref{equ1} are true for $n=1$ and $n=2$.
Let $n\geq3$ and let $\pi=\pi_{1}\pi_{2}\cdots\pi_{n}\in\cs_n$. For any $i=1,2,\ldots,\lrf{n/2}$, define a map $\varphi_{i}:\cs_{n}\rightarrow\cs_{n}$ by
$$\varphi_{i}(\pi)=\pi_{1}\pi_{2}\cdots\pi_{2i}\pi_{2i-1}\cdots\pi_{n},$$
i.e., $\varphi_{i}(\pi)$ is obtained by swapping $\pi_{2i}$ with $\pi_{2i-1}$ in $\pi$.
Obviously, the map $\varphi_{i}$ is an involution, $i=1,2,\ldots,\lrf{n/2}$, and $\varphi_{i}$ and $\varphi_{j}$ commute for all $i,j\in\{1,2,\ldots,\lrf{n/2}\}$.
For any subset $S\subseteq \{1,2,\ldots,\lrf{n/2}\}$, we define a map $\varphi_{S}:\cs_{n}\rightarrow\cs_{n}$ by
$$\varphi_{S}(\pi)=\prod_{i\in S}\varphi_{i}(\pi).$$

The group $\mathbb{Z}_{2}^{\lrf{n/2}}$ acts on $\cs_{n}$ via the maps $\varphi_{S},S\subseteq \{1,2,\ldots,\lrf{n/2}\}$.
For any $\pi\in\cs_n$, let $\Orb^*(\pi)$ denote the orbit including $\pi$ under the group action. There is a unique permutation in $\Orb^*(\pi)$,
denoted by $\hat{\pi}$, such that
$$\hat{\pi}_{1}<\hat{\pi}_{2},~ \hat{\pi}_{3}<\hat{\pi}_{4},~ \ldots,~ \hat{\pi}_{2\lrf{n/2}-1}<\hat{\pi}_{2\lrf{n/2}}.$$
It is not hard to prove that $\odes(\hat{\pi})=0$ and $\odes(\varphi_{S}(\hat{\pi}))=|S|$ for any $S\subseteq \{1,2,\ldots,\lrf{n/2}\}$. Then
$$\sum_{\sigma\in \Orb^*(\pi)}p^{\odes(\sigma)}=(1+p)^{\lrf{\frac{n}{2}}}.$$

Let $\cs^*_{n}$ consist of all the permutations in $\cs_{n}$ such that
$$\pi_{1}<\pi_{2},~ \pi_{3}<\pi_{4},~ \ldots,~ \pi_{2\lrf{n/2}-1}<\pi_{2\lrf{n/2}}.$$
The cardinality of the set $\cs^*_{n}$ is
$$\binom{n}{2}\binom{n-2}{2}\cdots\binom{n+2-2\lrf{\frac{n}{2}}}{2}=\frac{n!}{2^{\lrf{\frac{n}{2}}}}.$$
Then
$$\sum_{\pi\in\cs_{n}}p^{\odes{(\pi)}}=A_n(p,1)=\frac{n!}{2^{\lrf{\frac{n}{2}}}}(1+p)^{\lrf{\frac{n}{2}}}.$$

Similarly, for any $i=1,2,\ldots,\lrf{(n-1)/2}$, we define a map $\phi_{i}:\cs_{n}\rightarrow\cs_{n}$ by
$$\phi_{i}(\pi)=\pi_{1}\cdots\pi_{2i+1}\pi_{2i}\cdots\pi_{n},$$
i.e., $\phi_{i}(\pi)$ is obtained by swapping $\pi_{2i}$ with $\pi_{2i+1}$ in $\pi$.
Obviously, the map $\phi_{i}$ is an involution, $i=1,2,\ldots,\lrf{(n-1)/2}$, and $\phi_{i}$ and $\phi_{j}$ commute for all $i,j\in\{1,2,\ldots,\lrf{(n-1)/2}\}$.
For any subset $S\subseteq \{1,2,\ldots,\lrf{(n-1)/2}\}$, we define a map $\phi_{S}:\cs_{n}\rightarrow\cs_{n}$ by
$$\phi_{S}(\pi)=\prod_{i\in S}\phi_{i}(\pi).$$

The group $\mathbb{Z}_{2}^{\lrf{(n-1)/2}}$ acts on $\cs_{n}$ via the maps $\phi_{S},S\in [\lrf{(n-1)/2}]$.
For any $\pi\in\cs_n$, let $\Orb^{**}(\pi)$ denote the orbit including $\pi$ under the group action. There is a unique permutation in $\Orb^{**}(\pi)$,
denoted by $\bar{\pi}$, such that
$$\bar{\pi}_{2}<\bar{\pi}_{3},~ \bar{\pi}_{4}<\bar{\pi}_{5},~ \ldots,~ \bar{\pi}_{2\lrf{(n-1)/2}}<\bar{\pi}_{2\lrf{(n-1)/2}+1}.$$
It is easily obtained that $\edes(\bar{\pi})=0$ and $\edes(\phi_{S}(\bar{\pi}))=|S|$ for any $S\subseteq \{1,2,\ldots,\\
\lrf{(n-1)/2}\}$. Then
$$\sum_{\sigma\in \Orb^{**}(\pi)}q^{\edes(\sigma)}=(1+q)^{\lrf{\frac{n-1}{2}}}.$$

Let $\cs^{**}_{n}$ consist of all the permutations in $\cs_{n}$ such that
$$\pi_{2}<\pi_{3},~ \pi_{4}<\pi_{5},~ \ldots,~ \pi_{2\lrf{(n-1)/2}}<\pi_{2\lrf{(n-1)/2}+1}.$$
The cardinality of the set $\cs^{**}_{n}$ is
\begin{equation*}
\begin{cases}
\binom{n}{2}\binom{n-2}{2}\cdots\binom{n+2-2\lrf{\frac{n-1}{2}}}{2}=\frac{n!}{2^{\lrf{\frac{n-1}{2}}}} & \text{if $n$ is odd,}\\
2\binom{n}{2}\binom{n-2}{2}\cdots\binom{n+2-2\lrf{\frac{n-1}{2}}}{2}=\frac{n!}{2^{\lrf{\frac{n-1}{2}}}} & \text{if $n$ is even.}
\end{cases}
\end{equation*}
Then
$$\sum_{\pi\in\cs_{n}}q^{\edes(\pi)}=A_n(1,q)=\frac{n!}{2^{\lrf{\frac{n-1}{2}}}}(1+q)^{\lrf{\frac{n-1}{2}}}.$$
\end{proof}

\section{Remarks}\hspace*{\parindent}
The set of palindromic bivariate polynomials of darga $k$ is a vector space with gamma basis
$$\mathcal{B}_{k}=\{(pq)^{i}(p+q)^{j}(1+pq)^{k-2i-j}|i,j\geq0,2i+j\leq k\}.$$
Thus the refined Eulerian polynomials $\widetilde{A}_n(p,q)$, $n=1,2,\ldots,$
can be expanded in terms of the gamma basis $\mathcal{B}_{\lrf{\frac{n}{2}}}$. For example,
\begin{align*}
\widetilde{A}_{1}(p,q)&=A_{1}(p,q)=1,\\
\widetilde{A}_{2}(p,q)&=(1+q)A_{2}(p,q)=(1+q)(1+p)=1+p+q+pq\\
&=(1+pq)+(p+q),\\
\widetilde{A}_3(p,q)&=A_{3}(p,q)=1+2p+2q+pq=(1+pq)+2(p+q),\\
\widetilde{A}_{4}(p,q)&=(1+q)A_{4}(p,q)=(1+q)(1+6p+5q+5p^2+6pq+p^{2}q)\\
&=1+6p+6q+5p^2+12pq+5q^2+6p^{2}q+6pq^{2}+p^{2}q^{2}\\
&=(1+pq)^2+6(p+q)(1+pq)+5(p+q)^2,\\
\widetilde{A}_{5}(p,q)&=A_{5}(p,q)=1+13p+13q+16p^2+34pq+16q^2+13p^{2}q+13pq^{2}+p^{2}q^{2}\\
&=(1+pq)^2+13(p+q)(1+pq)+16(p+q)^2,\\
\widetilde{A}_{6}(p,q)&=(1+q)A_{6}(p,q)\\
&=(1+q)(1+29p+28q+89p^2+152pq+61q^2\\
&~~~+61p^{3}+152p^{2}q+89pq^{2}+28p^{3}q+29p^{2}q^{2}+p^{3}q^{2})\\
&=1+29p+29q+89p^2+89q^2+181pq+61p^3+241p^{2}q\\
&~~~+241pq^{2}+61q^3+181p^{2}q^{2}+89p^{3}q+89pq^{3}+29p^{3}q^{2}+29p^{2}q^{3}+p^{3}q^{3}\\
&=(1+pq)^3+29(p+q)(1+pq)^2+89(p+q)^2(1+pq)+61(p+q)^3.
\end{align*}

We conjecture that for any $n\geq 1$, all $c_{j}$ are positive integers in the following expansion
$$\widetilde{A}_n(p,q)=\sum_{j=0}^{\lrf{\frac{n}{2}}}c_{j}(p+q)^{j}(1+pq)^{\lrf{\frac{n}{2}}-j}.$$

\section*{Acknowledgment}\hspace*{\parindent}
I am grateful to my advisor Prof. Yi Wang for his valuable comments and suggestions. I also would like to thank the referee for his/her careful reading and many helpful suggestions.

\end{document}